\documentclass[a4paper, 10pt]{amsart}
\usepackage[top=80pt,bottom=65pt,left=70pt,right=70pt]{geometry}
\usepackage{graphicx, eepic}
\usepackage{times}
\normalfont
\usepackage{bbm}

\usepackage{amsthm,amsmath,amsfonts,amssymb}
\usepackage{hyperref}
\hypersetup{colorlinks}
\hypersetup{
    bookmarks=true,         
    unicode=false,          
    pdftoolbar=true,        
    pdfmenubar=true,        
    pdffitwindow=false,     
    pdfstartview={FitH},    
    pdftitle={My title},    
    pdfauthor={Author},     
    pdfsubject={Subject},   
    pdfcreator={Creator},   
    pdfproducer={Producer}, 
    pdfkeywords={keyword1} {key2} {key3}, 
    pdfnewwindow=true,      
    colorlinks=true,       
    linkcolor=red,          
    citecolor=red,        
    filecolor=red,      
    urlcolor=blue           
}
\usepackage{multirow, url}
\usepackage{color}
\usepackage[all]{xy}
\DeclareMathAlphabet{\mathscr}{OT1}{pzc}{m}{it} 

\newtheorem{thm}{Theorem}[section]
\newtheorem{cor}[thm]{Corollary}
\newtheorem{lem}[thm]{Lemma}
\newtheorem{prop}[thm]{Proposition}

\newtheorem{conj}[thm]{Conjecture}

\theoremstyle{definition}
\newtheorem{defn}[thm]{Definition}
\theoremstyle{remark}
\newtheorem{rem}[thm]{Remark}

\newtheorem{exam}[thm]{Example}

\numberwithin{equation}{section} \numberwithin{table}{section}

\newcommand{\zell}{{\Z/{\ell\Z}}}
\newcommand{\muell}{{\mu_{\ell}}}

\newcommand{\GQ}{{\mathrm{Gal}(\overline{\mathbb{Q}}/{\mathbb{Q}})}}

\newcommand{\Gm}{{\mathbb{G}}_{\mathrm{m}}}
\newcommand{\inj}{\hookrightarrow}
\newcommand{\surj}{\twoheadrightarrow}
\newcommand{\arsub}{\ar@{}[r]|-*[@]{\subset}}
\newcommand{\arsup}{\ar@{}[r]|-*[@]{\supset}}
\newcommand{\arcap}{\ar@{}[d]|-*[@]{\subset}}
\newcommand{\arcup}{\ar@{}[u]|-*[@]{\subset}}

\renewcommand{\pmod}[1]{{~(\mathrm{mod}~{#1})}}

\newcommand{\F}{{\mathbb{F}}}

\newcommand{\Q}{{\mathbb{Q}}}
\newcommand{\Z}{{\mathbb{Z}}}

\newcommand{\C}{{\mathbb{C}}}
\newcommand{\T}{{\mathbb{T}}}

\newcommand{\bP}{{\mathbb{P}}}

\newcommand{\fh}{{\mathfrak{h}}}

\newcommand{\m}{{\mathfrak{m}}}

\newcommand{\cI}{{\mathcal{I}}}

\newcommand{\cT}{{\mathcal{T}}}

\newcommand{\Qbar}{\overline{\mathbb{Q}}}

\newcommand{\Ext}{{\mathrm{Ext}}}
\newcommand{\Hom}{{\mathrm{Hom}}}
\newcommand{\Gal}{{\mathrm{Gal}}}
\newcommand{\Div}{{\mathrm{Div}}}

\newcommand{\End}{{\mathrm{End}}}

\newcommand{\GL}{{\mathrm{GL}}}

\newcommand{\Frob}{{\mathrm{Frob}}}

\newcommand{\Ver}{{\mathrm{Ver}}}
\newcommand{\old}{{\mathrm{old}}}
\newcommand{\new}{{\mathrm{new}}}
\newcommand{\Ta}{{\mathrm{Ta}}}

\newcommand{\upto}{{up to products of powers of 2 and 3}}
\newcommand{\modl}{{\pmod {\ell}}} 
\mathchardef\hyp="2D

\newcommand{\mat}[4]{
 \left(  \begin{smallmatrix} #1 & #2 \\ #3 & #4 \end{smallmatrix} \right)}

\newcommand{\vect}[2]{
 \left(  \begin{smallmatrix} #1 \\ #2 \end{smallmatrix} \right)}

\newcommand{\br}[1]{\langle #1 \rangle}

\newcommand{\exclude}[1]{}

\begin{document}                                                                          

\title{The index of an Eisenstein ideal and multiplicity one}
\author{Hwajong Yoo}
\address{Center for Geometry and Physics, Institute for Basic Science (IBS), Pohang, Republic of Korea 37673}
\email{hwajong@gmail.com}

\begin{abstract}
Mazur's fundamental work on the Eisenstein ideal of prime level has a variety of arithmetic applications. In this article, we generalize some of his work to square-free level. More specifically, we attempt to compute the index of an Eisenstein ideal and 
the dimension of the $\m$-torsion of the modular Jacobian variety, where $\m$ is an Eisenstein maximal ideal. In many cases, the dimension of the $\m$-torsion is 2, in other words, a multiplicity one theorem holds. 
\end{abstract}
\thanks{}

\maketitle
\setcounter{tocdepth}{1}
\tableofcontents

\section{Introduction}
In the early 20th century, Ramanujan found the following congruences:
$$
\tau(p) \equiv 1+ p^{11} \pmod {691}
$$
for any prime $p$, where $\tau(p)$ is the $p^{\rm{th}}$ Fourier coefficient of the cusp form $\Delta(z)= q \prod_{n \geq 1} (1-q^n)^{24}$ of weight 12 and level 1. Many mathematicians further found congruences between $\tau(p)$ and some arithmetic functions (such as $\sigma_k(p)=1+p^k$) modulo small prime powers. Serre and Swinnerton-Dyer \cite{Se72, Sw73} recognized that these congruences between Eisenstein series and cusp forms can be understood by making use of Galois representations associated to $\Delta$, and they determined all possible congruences of $\tau(p)$.

In the case of weight 2, Mazur discussed the Eisenstein ideal, which detect the congruences between Eisenstein series and cusp forms. (An ideal of a Hecke ring is called \textit{Eisenstein} if it contains $T_r-r-1$ for all primes $r$ not dividing the level.) 
For a prime $N$, he proved Ogg's conjecture (Conjecture \ref{ogg} below)
via a careful study of subgroups of the Jacobian variety $J_0(N)$ annihilated by the Eisenstein ideal. 
More precisely, in his paper \cite{M77}, he proved that $\T(N)/I \simeq \Z/{n\Z}$, where $\T(N)$ is the Hecke ring of level $N$, $I$ is the Eisenstein ideal of $\T(N)$, and $n$ is the numerator of $\frac{N-1}{12}$. Moreover, for each prime $\ell \mid n$, he further proved that $\dim J_0(N)[\m]=2$, where $\m$ is an Eisenstein maximal ideal generated by $I$ and $\ell$, and
$$
J_0(N)[\m]:= \{ x \in J_0(N)(\overline{\Q})~:~ Tx=0~\text{ for all } T \in \m \}.
$$
Using this result he proved a classification theorem for the rational torsion subgroups of elliptic curves over $\Q$.

After his work, the dimension of $J_0(N)[\m]$ has been studied by several mathematicians for the case that $\m$ is non-Eisenstein. 
Assume that $\m$ is a non-Eisenstein maximal ideal of the Hecke ring $\T$ of level $N$. Then, by Boston-Lenstra-Ribet \cite{BLR91}, $J_0(N)[\m] \simeq V^{\oplus a}$, where $V$ is the underlying irreducible module for the Galois representation $\rho_{\m} : \GQ \rightarrow \GL(2, \T/{\m})$ associated to $\m$. In most cases, for instance, when the characteristic of $\T/{\m}$ is prime to $2N$, the multiplicity $a$ is one. 
A slight generalization of this multiplicity one result played an important role in the work of Wiles on Fermat's Last Theorem. On the other hand, for an Eisenstein maximal ideal $\m$, the dimension of $J_0(N)[\m]$ has not been discussed when the level $N$ is not prime. 

In this paper, we generalize some part of results of Mazur \cite{M77} to square-free level. 
In \textsection 2, we introduce some Eisenstein modules that are used later. 

In \textsection 3, we attempt to compute the index of an Eisenstein ideal of square-free level. More precisely, let $N$ be a square-free integer. Then, any Eisenstein maximal ideal of $\T(N)$ contains 
$$
I_M :=( U_p-1, ~U_q-q,~ T_r-r-1 ~:~ \text{for primes} ~p \mid M,~ q \mid {N/M},~ \text{and}~ r \nmid N )
$$ 
for some divisor $M>1$ of $N$ by \cite[Lemma 2.1]{Yoo3}. 
For a square-free integer $K=\prod p_i$, let $\varphi(K)=:\prod (p_i-1)$ and $\psi(K):=\prod(p_i+1)$.
Then, we prove the following theorem.

\begin{thm}\label{thm : 1.1}
For any prime $y \nmid 2N$, we have 
$$
\T(N)/I_M \otimes \Z_{y} \simeq \Z/{m\Z} \otimes \Z_{y},
$$ 
where $m$ is the numerator of $\frac{\varphi(N)\psi(N/M)}{24}$. 
\end{thm}
(For the analogous theorem about the index of a ``new" Eisenstein ideal of $\T(N)$, see Theorem \ref{thm:neweisen}.)

In \cite{Yoo3}, the author extended Theorem \ref{thm : 1.1} to $y \mid 2N$ when $N/M$ is odd and greater than 1, and to $y \mid N$ when $N/M$ is even or $N=M$. When $y=2$, if $N/M$ is even or $N=M$, then we do not know what happens.

The following conjecture is presumably known to experts, but we could not find the origin.

\begin{conj}[Generalized Ogg's conjecture] \label{ogg}
The rational torsion subgroup of $J_0(N)$ is the cuspidal group.
\end{conj}  
(About the definition of the cuspidal group, see \textsection \ref{sec:cuspidalgroup}.)
Most of this conjecture was proved by Ohta \cite[Theorem]{Oh14}. For $N=pq$, using the result of \cite{Yoo3} on the index of an Eisenstein ideal, the author independently proved a result on this conjecture, overlapping Ohta's but not completely subsumed by it \cite[Theorem 1.2]{Yoo5}.

For an Eisenstein maximal ideal $\m$ of level $N$, we define
$$S_{\m} : = \text{ the set of primes at which } J[\m] \text{ is ramified};$$
$$S_N :=  \text{ the set of prime divisors of } N; $$
$$s_0(\m) := \#\{ p \mid N~:~p ~\equiv ~1 \pmod {\m} \};$$
$$s(\m) := \#\{ p \mid N~:~U_p \equiv ~1 \pmod {\m} \} \quad\text{and}$$
$$\varpi_0(\m) := \begin{cases}
s(\m)~\quad\text{ if }~~s(\m)=s_0(\m) \\ 
~~0 \quad\quad\quad ~\text{otherwise}. \\
\end{cases}$$
For a finite set $S$ of primes, we define
$$
\varpi_{\ell}(S):= \# \{ p \in S~:~p \equiv \pm 1 \modl  \}.
$$

In \textsection 4, we study the dimension of $J_0(N)[\m]$ for an Eisenstein maximal ideal $\m$ of residue characteristic $\ell \nmid 6N$. So, we assume that $\ell \nmid 6N$. Note that, by Theorem \ref{thm : 1.1}, $\varpi_{\ell}(S_N) \geq 1$.

\begin{thm}\label{thm:multi}
Assume that $\varpi_{\ell}(S_N)=1$. Then, we have  
$$
\dim J_0(N)[\m] = 2.
$$
In other words, a multiplicity one theorem holds for an Eisenstein maximal ideal $\m$ if $\varpi_{\ell}(S_N)=1$.
\end{thm} 

We further prove a bound on $\dim J_0(N)[\m]$ involving the two numbers $\varpi_0(\m)$ and $\varpi_{\ell}(S_{\m})$.

\begin{thm}\label{thm-bound}
We have
$$
\max\{1+\varpi_0(\m), 2\} \leq \dim J_0(N)[\m] \leq 1+\varpi_0(\m)+\varpi_{\ell}(S_{\m}).
$$
\end{thm}
Note that $\varpi_{\ell}(S_{\m})\leq \varpi_{\ell}(S_N)$ since $S_{\m} \subseteq S_N \cup \{ \ell \}$, so we can have an explicit bound on the dimension without further information about ramification of $J_0(N)[\m]$. Recently, Ribet and the author proved that the above upper bound is \textit{optimal} if $\varpi_0(\m)=0$. In other words, if $\varpi_0(\m)=0$, then
$$
\dim J_0(N)[\m]=1 + \varpi_{\ell}(S_{\m}).
$$
The author expects that the above upper bound is optimal unless $\varpi_0(\m)=1$.

Let us briefly remark upon the proof of the above theorem. First, we generalize some of Mazur's work to analyze $J_0(N)[\m]$. Next, we use Vatsal's work \cite{Va05} on the Shimura subgroup to understand how many copies of $\muell$'s can be contained in $J_0(N)[\m]$. Last, by the recent work of Brumer and Kramer \cite{BK14} on the extension group of $\muell$ by $\zell$ (over $\Z[1/N]$), we get an upper bound on $\dim J_0(N)[\m]$.
 
Theorem \ref{thm:multi} can be applied to the study of the structure of $J_0(N)[\m]$. If a multiplicity one theorem holds,
it gives one of the models of the Galois representation $\rho_{\m}$ associated to $\m$. Moreover, $J_0(N)[\m]$ can only be ramified at $\ell$ and at most a single prime divisor $p$ of $N$ such that $p \equiv \pm 1 \modl$. (See the proof of Theorem \ref{multi}.)
In the case when $N$ is the product of two distinct primes $p$ and $q$, we prove a more precise result on
$\dim J_0(N)[\m]$. This result also gives the description of the Galois module $J_0(N)[\m]$ by computing its dimension as a vector space over $\T/{\m}\simeq \F_{\ell}$. 

\subsubsection*{Acknowledgements}
The author would like to thank his advisor Kenneth Ribet. This paper would not exist if it were not for his inspired suggestions and his constant enthusiasm for the work. The author would like to thank Chan-Ho Kim, Sara Arias-de-Reyna, Sug Woo Shin, and Gabor Wiese for many suggestions toward the correction and improvement of this paper. The author would like to thank the anonymous referee for careful reading and a number of suggestions to improve the paper.
The author would also like to thank Samsung scholarship for supporting him during the course of graduate research.

\subsection{Notation}\label{sec:not}
Let $X_0(N)$ be the modular curve for $\Gamma_0(N)$ and let $J_0(N)$ be the Jacobian variety of $X_0(N)$. By Igusa \cite{Ig59}, Deligne-Rapoport \cite{DR73}, Katz-Mazur \cite{KM85}, and Raynaud \cite{Ra70}, there exists the N\'eron model of $J_0(N)$ over $\Z$, which we denote it by $J_0(N)_{/{\Z}}$. We denote by $J_0(N)_{/{\F_p}}$ the special fibers of $J_0(N)_{/{\Z}}$ over $\F_p$. (Since we will assume that $N$ is square-free later, all the special fiber of $J_0(N)_{/{\Z}}$ at bad primes come from the Deligne-Rapoport model.) We denote by $M_2(\Gamma_0(N))$ (resp. $S_2(\Gamma_0(N))$) the space of modular (resp. cusp) forms of weight 2 and level $\Gamma_0(N)$ over $\C$.

From now on, we assume that $\ell$ is a prime larger than 3. And we assume that the level $N$ is square-free and prime to $\ell$. 

For a square-free number $N$, we define
$$
\varphi(N):=\prod\limits_{p\mid N~\text{primes}} (p-1)\quad\mathrm{and} \quad\psi(N):= \prod\limits_{p \mid N ~\text{primes}} (p+1).
$$

For any group or module $X$, we denote by $\End(X)$ its endomorphism ring. We denote by $\zell$ the constant group scheme of order $\ell$ (over $\Z$). We denote by $\muell$ the multiplicative group scheme of order $\ell$ (over $\Z$). 
For a finite set $S$ of primes, we denote by $\Ext_{S}(\muell, \zell)$ the group of extensions of $\muell$ by $\zell$ over the ring $\Z_S:=\Z[1/n]$, where $n=\prod_{p \in S} p$. 

\section{Eisenstein series and Eisenstein modules}\label{sec:Eisen}
\subsection{Hecke operators}\label{sec:heckeoperators} 
Throughout this section, we assume that $p$ is a prime not dividing $N$ and $q$ is a prime divisor of $N$. 
\subsubsection{Degeneracy maps on modular curves} 
For a field $K$ of characteristic not dividing $Np$, the points of $X_0(Np)$ over $K$ are  isomorphism classes of  the triples $(E, C, D)$, where $E$ is a (generalized) elliptic curve over $K$, $C$ is a cyclic subgroup of $E$ of order $N$, and $D$ is a cyclic subgroup of $E$ of order $p$. Similarly, the points of $X_0(N)$ over $K$ are isomorphism classes of the pairs $(E, C)$. We can consider natural maps between modular curves
$$
\xymatrix{
X_0(Np) \ar@<.5ex>[r]^-{\alpha_p} \ar@<-.5ex>[r]_-{\beta_p} & X_0(N),
}
$$
where $\alpha_p (E, C, D):= (E, C)$ and $\beta_p(E, C, D) := (E/D, (C + D)/D)$. In other words, the map $\alpha_p$ is ``forgetting level $p$ structure" and the map $\beta_p$ is ``dividing by level $p$ structure".
As a map $X_0(Np)(\C)\simeq \fh^*/{\Gamma_0(Np)} \rightarrow \fh^*/{\Gamma_0(N)}\simeq X_0(N)(\C)$ on the complex points, $\alpha_p$ (resp. $\beta_p$) sends $z$ to $z$ (resp. $pz$), where $\fh$ is the complex upper half plane and $\fh^*=\fh \bigcup \bP^1(\Q)$.

\subsubsection{Degeneracy maps on modular forms} 
The maps $\alpha_p$ and $\beta_p$ above induce maps
$\alpha_p^*$ and $\beta_p^*$ between cusp forms of weight two, respectively as follows:
$$
\xymatrix{
S_2(\Gamma_0(N)) \ar@<.5ex>[r]^-{\alpha_p^*} \ar@<-.5ex>[r]_-{\beta_p^*} & S_2(\Gamma_0(Np)),
}
$$
where $\alpha_p^* (f(\tau))=f(\tau)$ and $\beta_p^* (f(\tau))= pf(p\tau)$. On its Fourier expansions at $i\infty$, $\alpha_p^*(\sum a_n x^n)=\sum a_n x^n$ and $\beta_p^*(\sum a_n x^n) = p\sum a_n x^{pn}$, where $x=e^{2\pi i \tau}$. Note that we use the $x$-expansion instead of the $q$-expansion because we denote by $q$ a prime divisor of the level $N$. From now on, by \textit{the $x$-expansion of $f$} we mean the classical $q$-expansion of $f$. These maps can also be extended to $M_2(\Gamma_0(N))$ via the same formula on their $x$-expansions.

\subsubsection{Hecke operators on modular curves and modular forms} 
The above degeneracy maps induce maps between divisor groups of modular curves. More specifically, we have 
$$
\xymatrix{
\Div(X_0(N)) \ar@<.5ex>[r]^-{\alpha_p^*} \ar@<-.5ex>[r]_-{\beta_p^*} & \Div(X_0(Np)) \ar@<.5ex>[r]^-{\alpha_{p, *}} \ar@<-.5ex>[r]_-{\beta_{p, *}} & \Div(X_0(N)),
}
$$
where 
$$
\alpha_p^* (E, C) = \sum_{D \subset E[p]} (E, C, D),~~~
\beta_p^* (E, C) = \sum_{D \subset E[p]} (E/D, (C + D) /D, E[p]/D);
$$
$$
\alpha_{p, *}(E, C, D) = (E, C), ~~\text{ and }~~ \beta_{p, *}(E, C, D) = (E/D, (C + D)/D).
$$
In the summation of the above formula, $D$ runs through all cyclic subgroups of order $p$. We define $T_p$ acting on $\Div(X_0(N))$ to be $\alpha_{p, *}\circ \beta_p^*$ (Albanese) or 
$\beta_{p, *}\circ \alpha_p^*$ (Picard). In terms of divisors we have
$$
T_p(E, C) = \sum_{D \subset E[p]} (E/D, (C+D)/D),
$$
where $D$ runs through all cyclic subgroups of order $p$.
This map induces an endomorphism of the Jacobian $J_0(N)$, which is also denoted by $T_p$. 

Since $S_2(\Gamma_0(N))$ can be identified with the cotangent space at 0 of $J_0(N)$, 
$T_p$ acts on $S_2(\Gamma_0(N))$. The above definition is compatible with the action of Hecke operators on $S_2(\Gamma_0(N))$, which is (on their $x$-expansions)
$$
T_p(\sum a_n x^n) := \sum a_{np}x^n + p \sum a_n x^{np}.
$$
This action also extends to $M_2(\Gamma_0(N))$ by the same formula on their $x$-expansions.

\subsubsection{Atkin-Lehner operators and more on Hecke operators} 
Let $q$ be a prime divisor of $N$. Since $N$ is square-free, $q^2$ does not divide $N$. Thus, $A := N/q$ is prime to $q$. We have an endomorphism $w_q$ of $X_0(N)$ such that
$$
w_q (E, C, D) = (E/D, (C+D)/D, E[q]/D),
$$
where $E$ is a (generalized) elliptic curve, $C$ is a cyclic subgroup of $E$ of order $A$, and $D$ is a cyclic subgroup of order $q$. It induces an Atkin-Lehner involution on $J_0(N)$, which  is also denoted by $w_q$. There is also the Hecke operator $T_q^t$ in $\End(\Div(X_0(N)))$, which acts by
$$
T_q^t(E, C, D) = \sum_{L \subset E[q]} (E/L, (C+L)/L, E[q]/L),
$$
where $L$ runs through all cyclic subgroups of $E$ of order $q$, which are different from $D$. This operator also induces an endomorphism of $J_0(N)$ (via Albanese functoriality), which is also denoted by $T_q^t$. 

\begin{lem}\label{lem:hecke atkin}
As endomorphisms of $J_0(N)$, we have $T_q^t+w_q = \beta_q^* \circ \alpha_{q, *}$, where 
$$
\xymatrix{
J_0(N)  \ar@<.5ex>[r]^-{\alpha_{q, *}}\ar@<-.5ex>[r]_-{\beta_{q, *}}& J_0(A)  \ar@<.5ex>[r]^-{\beta_q^*} \ar@<-.5ex>[r]_-{\alpha_q^*} & J_0(N).
}
$$ 
\end{lem}

\begin{proof}
On $\Div(X_0(N))$, $\alpha_{q, *}(E, C, D)=(E, C)$ and hence
$$
\beta_q^* \circ \alpha_{q, *} (E, C, D) = \sum _{L \subset E[q]} (E/L, (C+L)/L, E[q]/L),
$$
where $L$ runs through all cyclic subgroups of $E$ of order $q$. It is equal to $(T_q^t+w_q)(E, C, D)$, therefore they induce the same map on $J_0(N)$. 
\end{proof}

\begin{rem} \label{rem:picardfunc}
Note that $\alpha_q^* \circ \beta_{q, *}=w_q (\beta_q^* \circ \alpha_{q, *})w_q$ and $T_q^t=w_qT_q w_q$, where $T_q$ is the transpose of $T_q^t$. If we use Picard functoriality of Jacobian varieties, the formula will be written as $T_q+w_q=\alpha_q^* \circ \beta_{q, *}$ 
(cf. \cite[p. 444-446]{R90}, \cite[\textsection 13]{MR91}). From now on, we will use the Picard functoriality.
\end{rem}

\subsubsection{Hecke algebras} 
For any positive integer $n$, we define $T_n$ as follows:
\begin{enumerate}
\item
$T_1=1$;

\item 
$T_{mn}=T_m T_n$ if $(m, n)=1$;

\item
$T_{p^k}=T_p T_{p^{k-1}}-pT_{p^{k-2}}$ for $k\geq 2$.
\end{enumerate}

We define $\T(N)$ as the $\Z$-subalgebra of $\End(J_0(N))$ generated by all $T_n$.
Note that $\T(N)$ is finite over $\Z$. Therefore all maximal ideals of $\T(N)$ are of finite index. 
We often denote by $U_q$ the Hecke operator $T_q$ for prime divisors $q$ of the level $N$ to emphasize the difference.

\subsection{Eisenstein series of weight 2 for $\Gamma_0(N)$}\label{sec:Eisensteinseries}
The space $M_2(\Gamma_0(N))$ of modular forms naturally decomposes into its subspace of cusp forms $S_2(\Gamma_0(N))$ 
and the \textit{Eisenstein space} $E_2(\Gamma_0(N))$, which is isomorphic to the quotient space $M_2(\Gamma_0(N))/{S_2(\Gamma_0(N))}$. We can pick a natural basis of $E_2(\Gamma_0(N))$ that consists of eigenforms for all Hecke operators. Since the number of cusps of $X_0(N)$ is $2^t$, where $t$ is the number of distinct prime divisors of $N$, the dimension 
of $E_2(\Gamma_0(N))$ is $2^t-1$. For more detail, see \cite[Chap.4]{DS05}.

\begin{defn}
We define $e$ to be the normalized Eisenstein series of weight $2$ and level $1$ whose $x$-expansion is
$$
-\frac{1}{24} + \sum_{n=1}^{\infty} \sigma(n)x^n,
$$
where $\sigma(n)=\sum\limits_{d\mid n, d>0} d$.
\end{defn}

\begin{rem}
Note that though $e$ is an eigenfunction for all Hecke operators, 
$e$ is not a classical modular form. And $e \modl $ is not a mod $\ell$ modular form of weight two (for a prime $\ell>3$), which means that it cannot be expressed as a sum of mod $\ell$ modular forms of weight two of any level prime to $\ell$. (In other words, the filtration of $e$ is $\ell+1$, not 2.) 
About this fact, see \cite{M77, Se72, Sw73}.  
\end{rem}

With the above function $e$, we can make Eisenstein series of weight two and level $N$ by raising the level.

\begin{defn}
For any modular form $g$ of weight $2$ and level $N$, and a prime $p$ not dividing $N$, we define 
$$
[p]^+(g)(z):=(\alpha_p^*-\beta_p^*)(g)= g(z) - p g(pz) \quad\mathrm{and}
$$
$$
[p]^-(g)(z):=(\alpha_p^*-\beta_p^*/p)(g) = g(z) - g(pz),
$$
where $\alpha_p^*$ and $\beta_p^* : M_2(\Gamma_0(N)) \rightarrow M_2(\Gamma_0(Np))$ are the two degeneracy maps in the previous section. 
\end{defn}

We want to apply the above maps to $e$. Since $e$ is not a genuine modular form, $[p]^+(e)$ or $[p]^-(e)$ may not have a proper meaning. However, $e(z)-pe(pz)$ becomes a genuine modular form of weight $2$ and level $p$, and $-24[e(z)-pe(pz)]$ is a modular form $e'$ on \cite[p. 78]{M77}. 
Also, $e(z)-pe(pz)$ is an eigenform for all Hecke operators. Therefore, we define $[p]^+(e)(z)$ by the same formula as above. (However, $[p]^-(e)$ is not a genuine modular form.) For more detail, see \cite[Example 2.2.6]{DI93} or \cite[\textsection 2.3]{Oh14}. Note that $U_p$ acts as $1$ on $[p]^+(e)$.

\begin{prop}
Let $g$ be an Eisenstein series of weight $2$ and level $N$ that is an eigenform for all Hecke operators.
Then, for a prime $p$ not dividing $N$, $[p]^+(g)$ and $[p]^-(g)$ are Eisenstein series of weight $2$ and level $Np$ such that the eigenvalues of $U_p$ are $1$ and $p$, respectively. 
\end{prop}

\begin{proof}
On the $p$-old subvariety of $J_0(Np)$, $U_p$ and $T_p$ satisfy the following equality, where $U_p$ (resp. $T_p$) denotes the $p^{\text{th}}$ Hecke operator in $\T(Np)$ (resp. $\T(N)$), 
$$
U_p \begin{pmatrix}
x \\ y 
\end{pmatrix}=
\begin{pmatrix}
T_p & p \\ -1 & 0
\end{pmatrix}
\begin{pmatrix}
x \\ y
\end{pmatrix}
$$
(cf. ``Formulaire" 4 in \cite{R89}).
The first (resp. second) row maps into $J_0(Np)$ by the map $\alpha_p^*$ (resp. $\beta_p^*$). Note that this formula also works on the whole space of modular forms, not just on the space of cusp forms.
Since on $g$, $T_p$ acts by $p+1$, $U_p$ acts by 1 on $\vect{x}{-x}$ and by $p$ on $\vect{px}{-x}$. Thus, the result follows.
\end{proof}

\begin{rem}
Since two degeneracy maps between $X_0(Np)$ and $X_0(N)$ commute with other Hecke operators $T_n$ if $(n, p)=1$, $[p]^+(g)$ and $[p]^-(g)$ are also eigenforms for all Hecke operators $T_n$ with $(n, p)=1$.
\end{rem}

\begin{rem}\label{uptp}
On the $p$-old subvariety of $J_0(Np)$, $U_p$ satisfies a quadratic equation
$$
X^2 - T_p X + p = 0
$$
by the Cayley-Hamilton theorem.
\end{rem}

\begin{defn}
For $1 \leq s \leq t$, let $N=\prod\limits_{i=1}^{t} p_i$ and $M=\prod\limits_{j=1}^{s} p_j$. We define
$$
E_{M, N}:=[p_t]^- \circ \cdots \circ [p_{s+1}]^- \circ [p_s]^+ \circ \cdots \circ [p_1]^+(e).
$$
\end{defn}

Since $M>1$, $[p_1]^+(e)$ is a genuine Eisenstein series of weight $2$ and level $p_1$, which is an eigenform for all the Hecke operators. Therefore $E_{M, N}$'s are Eisenstein series of weight $2$ and level $N$, which are eigenforms for all the Hecke operators by the above proposition.
For given $N=\prod_{i=1}^{t} p_i$ and each $1 \leq s \leq t$, there are 
$\left(
\begin{smallmatrix}
t \\ s
\end{smallmatrix} \right)
$
different choices for $M$. Thus, the number of all possible $E_{M, N}$ is $2^t -1$. Since they are all eigenforms for the Hecke operators and their eigensystems are different, they form a basis of $E_2(\Gamma_0(N))$.

Later we will use these modular forms to compute the index of an Eisenstein ideal. To do that, we need an information about constant terms of Fourier expansions of an Eisenstein series at various cusps, in particular, at 0 and at $i\infty$. Recall Proposition 3.34 in \cite{FJ95}. (About unfamiliar notations in the proposition, see \textit{loc. cit.})

\begin{prop}[Faltings-Jordan]\label{FJ} Suppose that $N=p N'$ with $(N', p)=1$. Let $g$ be a modular form of weight $k$ and level $N'$, so $w_{p} g$ has level $N$. 
\begin{enumerate}
\item The modular form $(\alpha(p)-w_{p})g$ has constant term $\alpha(p)(1-p^{k-1})a_0(g;c)$ at a $p$-multiplicative cusp $c$,
$\alpha(p)(1-1/{p})a_0(g;c)$ at a $p$-etale cusp $c$.
\item The modular form $(\beta(p){p}^{k-1}-w_{p})g$ has constant term $0$ at a \\
$p$-multiplicative cusp, $({p}^{k-1}\beta(p)-\alpha(p)/{p})a_0(g;c)$ at a $p$-etale cusp $c$.
\end{enumerate}
\end{prop}
Here $a_0(g; c)$ denotes the constant term of the Fourier expansion of $g$ at the cusp $c$. And $\alpha$ and $\beta$ denote the Dirichlet characters, which are the trivial characters in our situation. Also, $k=2$ and $w_{p}g(z)=p g(pz)=\beta_{p}^*(g)(z)$. Using the above result, we compute the constant term of $E_{M, N}$ of level $N$ at the multiplicative cusp 
$i\infty$ and the etale cusp $0$.
First recall that $E_{p, p}$ has constant term $-\frac{1-p}{24}$ at $i\infty$ and $-\frac{p-1}{24p}$ at $0$ (cf. \cite[p. 78]{M77}). 

\begin{prop} \label{Eisensteinseries}
The constant term of $E_{M, N}$ at $i\infty$ is either $0$ if $M\neq N$  or $(-1)^{t+1}\frac{\varphi(N)}{24}$ if $M=N$. Its constant term at $0$ is $-\frac{\varphi(N)\psi(N/M)}{24N(N/M)}$.
\end{prop}

\begin{proof}
Since $[p]^+(g)(z)=g(z)-pg(pz)=(\alpha(p)-w_p)g$, its constant term at $i\infty$ (resp. at $0$) is $(1-p)a_0(g)$ (resp. $\frac{1}{p}(p-1)b_0(g)$),
where $a_0(g)$ (resp. $b_0(g)$) is the constant term of $g$ at $i\infty$ (resp. at $0$).
Moreover, because $[p]^-(g)(z)=g(z)-g(pz)=\frac{1}{p}(\beta(p)p-w_p)g$, its constant term at $i\infty$ (resp. at $0$) is $0$
 (resp. $\frac{1}{p^2}(p-1)(p+1)b_0(g)$). Thus, the result follows by induction. 
\end{proof}

\subsection{The cuspidal group of $J_0(N)$} \label{sec:cuspidalgroup}
Let $P_n$ be the cusp corresponding to $\vect 1n$ as in \cite{Og74}. 
(It corresponds to $\frac{1}{n}$ in $\bP^1(\Q)$, so $P_1=0$ and $P_N = i\infty$.)
Then the cusps of $X_0(N)$ are those of the form $P_n$ for all possible positive divisors $n$ of $N$. 
\textit{The cuspidal group of $J_0(N)$} is the group generated by the equivalence classes of the degree $0$ divisors $\sum_{d \mid N} a_d P_d$. We introduce some special elements of the cuspidal group of $J_0(N)$.

\begin{defn} 
As in the previous section, let $N=\prod\limits_{i=1}^{t} p_i$ and $M=\prod\limits_{j=1}^{s} p_j$ for some $1 \leq s \leq t$. We define
$$
C_{M, N} := \sum\limits_{n\mid M} (-1)^{\omega(n)} P_n = P_1 - (\sum\limits_{i=1}^s P_{p_i})+ \cdots + (-1)^s P_M \in J_0(N),
$$
where $\omega(n)$ is the number of distinct prime divisors of $n$. And we denote by $\langle C_{M, N} \rangle$ the cyclic subgroup of $J_0(N)$ generated by $C_{M, N}$.
\end{defn}

\begin{prop} \label{cusp}
On the group $\langle C_{M, N} \rangle$, $U_{p_i}$ acts by 1 for $1 \leq i \leq s$ and $U_{p_j}$ acts by $p_j$ for $s<j\leq t$. For any prime $r$ not dividing $N$, $T_r$ acts by $r+1$ on $\langle C_{M, N} \rangle$. The order of $\langle C_{M, N} \rangle$ is the numerator of $\frac{\varphi(N)\psi(N/M)}{24}$ up to powers of 2.
\end{prop}

\begin{proof}
First, recall that for a prime divisor $p$ of $N$, $U_p+w_p$ acts on $J_0(N)$ by $\alpha_p^*\circ \beta_{p, *}$, where $\alpha_p$ and  $\beta_p$ are the two degeneracy maps 
$$
\xymatrix{
X_0(N) \ar@<.5ex>[r]^-{\alpha_p} \ar@<-.5ex>[r]_-{\beta_p} & X_0(N/p)
}
$$
and $w_p$ is the Atkin-Lehner involution. (See Remark \ref{rem:picardfunc}.) For some $D \mid N$ with $p \nmid D$, $\alpha_{p, *}(P_D)=\alpha_{p, *}(P_{pD})=\beta_{p, *}(P_D)=\beta_{p, *}(P_{pD})=P_D$.
And $\alpha_p^*(P_D)= pP_D + P_{pD}$, $\beta_p^*(P_D)=P_D+pP_{pD}$. Moreover $w_p(P_D)=P_{pD}$, $w_p(P_{pD})=P_D$. (cf. \cite[\textsection 5]{R89}.)

Let $p=p_i$ for some $1 \leq i \leq s$. Then $\beta_{p, *}(C_{M, N})=0$, hence $(U_p+w_p)(C_{M, N})=\alpha_p^*\circ\beta_{p, *}(C_{M, N})=0$. Since $w_p(C_{M, N})=-C_{M, N}$, 
$U_p(C_{M, N})=C_{M, N}$. 

Let $q=p_j$ for some $s < j \leq t$. Then
$$
w_q(C_{M, N})= P_q - (\sum\limits_{i=1}^s P_{q p_i})+(\sum\limits_{i<j}^s P_{q p_i p_j})+ \cdots + (-1)^s P_{qM}(=:C^{(q)}_{M, N})
$$
and $\beta_{q, *}(C_{M, N}) = C_{M, N/q}$. Thus
$$
(U_q+w_q)(C_{M, N})=\alpha_q^* (C_{M, N/q}) = qC_{M, N} + C^{(q)}_{M, N} = qC_{M, N} + w_q(C_{M, N}),
$$
so $U_q(C_{M, N})=q C_{M, N}$.

Next, let $r \nmid N$ be a prime. Then $T_r = \alpha_{r,*} \circ \beta_r^*$ on $J_0(N)$, where $\alpha_r$ and $\beta_r$ are the two degeneracy maps
$$
\xymatrix{
X_0(Nr) \ar@<.5ex>[r]^-{\alpha_r} \ar@<-.5ex>[r]_-{\beta_r} & X_0(N).
}
$$
For any $D \mid N$, $\beta_r^*(P_D)=P_D + rP_{rD}$, so $\alpha_{r,*} \circ \beta_r^*(P_D)=(r+1)P_D$. Thus, $T_r(C_{M, N})=(r+1)C_{M, N}$.

For the order of $\br {C_{M, N}}$, see \cite[Theorem 3.1]{Yoo3}. 
\end{proof}

\subsection{The Shimura subgroup of $J_0(N)$} 
\textit{The Shimura subgroup of $J_0(N)$} is the kernel of the map $J_0(N) \rightarrow J_1(N)$, that is induced by
the natural covering $X_1(N) \rightarrow X_0(N)$. 
Since the covering group of $X_1(N) \rightarrow X_0(N)$ is $(\Z/{N\Z})^{\times}/{\{\pm 1\}}$, 
the covering group of the maximal \'etale subcovering of $X_1(N) \rightarrow X_0(N)$ is a quotient of $(\Z/{N\Z})^{\times}/{\{\pm 1\}}$. The Shimura subgroup is the Cartier dual of this covering group. When $N$ is prime, Mazur discussed it on \textsection 11 of Chapter II in \cite{M77}.
(In general, see the paper by Ling and Oesterl\'e \cite{LO91}.)

As before let $N=\prod\limits_{i=1}^t p_i$, and let $\Sigma_N$ be the Shimura subgroup of $J_0(N)$. 
By the Chinese remainder theorem, $(\Z/{N\Z})^* \simeq (\Z/{p_1\Z})^* \times \cdots \times (\Z/{p_t\Z})^*$. Similarly, we can decompose the Shimura subgroup as
$$
\Sigma_N \simeq \Sigma_{p_1} \times \cdots \times \Sigma_{p_t}.
$$
Each $\Sigma_{p_i}$ corresponds to the subcovering $X_1(p_i, N/{p_i}) \rightarrow X_0(N)$ of $X_1(N) \rightarrow X_0(N)$,
where $X_1(A, B)$ is the modular curve for the group $\Gamma_1(A) \cap \Gamma_0(B)$. 
Note that $\Sigma_{p_i}$ is cyclic of order $p_i-1$ \upto. 

\begin{prop}[Ling-Oesterl\'e] \label{Shimura}
On $\Sigma_{p_i}$, $U_{p_i}$ acts by 1, $U_{p_j}$ acts by $p_j$ for $j \neq i$, and $T_r$ acts by $r+1$ for primes $r \nmid N$.
\end{prop}
\begin{proof}
Ling and Oesterl\'e proved that $U_p$ acts on $\Sigma_N$ by $p$ for primes $p \mid N$ \cite[Theorem 6]{LO91}. In our case, since $\Sigma_{p_i}$ has order dividing $p_i-1$, the result follows from their work.
\end{proof}

\subsection{The component group of $J_0(N)$ (over $\F_p$)}
The component group of $J_0(N)$ (over $\F_p$) for square-free level $N$ is explained in the appendix of \cite{M77}.
Let $N=pA$, $(p, A)=1$, and $J:=J_0(N)$. Assume that $A$ is square-free. Then, by the results of Deligne-Rapoport \cite{DR73} and Raynaud \cite{Ra70}, 
$J_{\F_p}$ is an extension of a finite group $\Phi_p(J)$,
\textit{the component group of $J_{\F_p}$}, by a semiabelian variety $J^0$, \textit{the identity component
of $J_{\F_p}$}. Moreover $J^0$ is an extension of $J_0(A)_{/\F_p} \times J_0(A)_{/\F_p}$ by $T$, 
\textit{the torus of $J_{\F_p}$}. 

\begin{prop} \label{appendix}
The order of $\Phi_p(J)$ is equal to $(p-1)\psi(A)$ \upto, and $\Phi_p(J) = \Phi \oplus B$,
where $\Phi$ is generated by the image of $C_{p, N}=P_1-P_p$ and the order of $B$ divides
some product of powers of 2 and 3.
Moreover $\Frob_p$, the Frobenius endomorphism in characteristic $p$, acts by $-pw_p$ on $T$, where $w_p$ is the Atkin-Lehner operator defined in \textsection \ref{sec:heckeoperators}.
\end{prop}
\begin{proof}
This is the main result of the appendix in \cite{M77} by Mazur and Rapoport.
Since $P_p$ and $P_N$ lie in the same component of the N\'eron model of $X_0(pA)$ over $\F_p$, 
the elements $P_1-P_p$ and $P_1-P_N$ generate the same group $\Phi$ in \textit{loc. cit.}
\end{proof}

\begin{rem}\label{torus}
Since the Hecke action on $T$ factors through the $p$-new quotient of $\T(N)$ \cite[Proposition 3.7]{R90}, $U_p+w_p$ annihilates $T$. Therefore $\Frob_p$ acts by $pU_p$ on $T$.
\end{rem}

\begin{prop}\label{comp}
On $\Phi$, $U_p$ acts by $1$, $U_q$ acts by $q$ for $q \mid A$ primes, and $T_r$ acts by $r+1$
for $r \nmid N$ primes. Moreover for a prime $\ell>3$, $\Phi_p(J)[\ell]$, the $\ell$-torsion elements in 
$\Phi_p(J)$, is equal to $\Phi[\ell]$, and $\Phi[\ell] \simeq \zell$ as groups if $\ell \mid (p-1)\psi(A)$.
\end{prop}

\begin{proof}
Because the reduction map is Hecke equivariant, this is an easy consequence of Proposition \ref{cusp} and Proposition \ref{appendix}.
\end{proof}

\begin{rem}\label{cuspcomp}
Since the order of $C_{p, N}$ (resp. of $\Phi$) is $\varphi(N)\psi(A)$ (resp. $(p-1)\psi(A)$) \upto, the kernel of the map $\langle C_{p, N} \rangle \twoheadrightarrow \Phi$ is of order $\varphi(N/p)$ \upto. Therefore, if $\ell \nmid \varphi(N/p)$ and $\ell>3$, the order $\ell$ subgroup of $\langle C_{p, N} \rangle$ maps isomorphically into $\Phi[\ell] = \Phi_p(J)[\ell]$. 
\end{rem}

\subsection{Old and new}\label{old and new}
Let $N$ be a square-free integer and let $p$ be a prime divisor of $N$. 
Two degeneracy maps $\alpha_p$ and $\beta_p$ in \textsection \ref{sec:heckeoperators} induce a map
$$
\gamma_p^* : J_0(N/p) \times J_0(N/p) \rightarrow J_0(N)
$$
defined by $\gamma_p^*(x, y):=\alpha_p^*(x)+\beta_p^*(y)$. 
The image of $\gamma_p^*$ is called \textit{the $p$-old subvariety} and the quotient of $J_0(N)$ by the $p$-old subvariety is called \textit{the $p$-new quotient}. The matrix relation 
$U_p = \mat {\tau_p} p {-1} 0$, where $\tau_p$ is the $p^{\text{th}}$ Hecke operator on $J_0(N/p)$ makes $\gamma_p^*$ Hecke-equivariant. The image of $\T(N)$ in the endomorphism ring of the $p$-old subvariety (resp. $p$-new quotient) is called \textit{the $p$-old quotient} (resp. \textit{the $p$-new quotient}), which is denoted by $\T(N)^{p\hyp\old}$ (resp. $\T(N)^{p\hyp\new}$). A maximal ideal of $\T(N)$ is called \textit{$p$-old} (resp. \textit{$p$-new}) if it is still maximal in the $p$-old (resp. $p$-new) quotient. 

The dual of the $p$-new quotient in $J_0(N)$ by the auto duality of Jacobian varieties is called \textit{the $p$-new subvariety}. 
The subvariety of $J_0(N)$ generated by the $p$-old subvarieties for all prime divisors $p$ of $N$ is called \textit{the old subvariety}. The quotient of $J_0(N)$ by the old subvariety is called \textit{the new quotient}, which is denoted by $J_0(N)^{\new}$, and its dual is called \textit{the new subvariety}, which is denoted by $J_0(N)_{\new}$. The old and new subvarieties are stable under the Hecke actions. The image of $\T(N)$ in $\End(J_0(N)^{\new})$ is called \textit{the new quotient} of $\T(N)$, which is denoted by 
$\T(N)^{\new}$. A maximal ideal of $\T(N)$ is called \textit{new} if it is still maximal ideal in $\T(N)^{\new}$.
 
Note that any maximal ideal of $\T(N)$ is either $p$-old or $p$-new (or both). Also, $U_p+w_p=0$ in $\T(N)^{p\hyp\new}$ by Lemma \ref{lem:hecke atkin}, and hence $U_p^2=1$ in $\T(N)^{p\hyp\new}$. For another description of old and new, see \cite[\textsection 3]{R90}.

\section{The index of an Eisenstein ideal}
Recall that for an ideal of $\T(N)$ with $N$ square-free, we call it \textit{Eisenstein} if it contains $T_r-r-1$ for all primes $r$ not dividing the level $N$. Any Eisenstein maximal ideal of $\T(N)$ contains
$$
I_M:=(U_p-1, ~U_q-q, ~T_r-r-1 ~:~ \text{for all primes }~ p \mid M, ~q \mid N/M \text{ and }~ r \nmid N )
$$ 
for some $M>1$. In order to find a criterion for a ideal $\m=(\ell, ~I_M)$ to be maximal, it suffices to compute the index of $I_M$.
\subsection{Square-free level}
Let $N=\prod\limits_{i=1}^{t} p_i$ and $M=\prod\limits_{j=1}^{s} p_j$ for some $1 \leq s \leq t$. 
And let $\T:=\T(N)$. 

Throughout this section, we implicitly use the duality of the Hecke algebra and cusp forms over $\Z[1/N]$-algebras.  See \cite[chap II, \textsection 4]{M77} or \cite[\textsection 2]{Oh14}.

\begin{lem} \label{notz}
The quotient ring $\T/{I_M}$ is isomorphic to $\Z/{n\Z}$ for some integer $n>0$.
\end{lem}
\begin{proof}
The natural map $\Z \rightarrow \T/{I_M}$ is surjective, since, modulo $I_M$, the operators $T_p$ are all congruent to integers. 
Let $F:=\sum_{n \geq 1} (T_n \pmod {I_M}) x^n$.
We cannot have $\T/{I_M} = \Z$, for then $F$ would be the $x$-expansion of a cuspidal eigenform over $\C$, which contradicts the Ramanujan-Petersson bounds and hence $\T/{I_M} \simeq \Z/{n\Z}$ for some integer $n>0$.
\end{proof}

\begin{thm}\label{index}
For any prime $y \nmid 2N$, we have
$$
(\T/{I_M}) \otimes_{\Z} \Z_{y} \simeq (\Z/{m\Z}) \otimes_{\Z} \Z_{y},
$$
where $m$ is the numerator of $\frac{\varphi(N)\psi(N/M)}{24}$.
\end{thm}

\begin{proof}
Let $\T/{I_M} \simeq \Z/{n\Z}$ and let $y \nmid 2N$ be a prime. Let $n=y^a \times b$ and $m=y^c \times d$, where $(y, bd)=1$. Let $J=(y^a, {I_M})$. Then, $\T/J \simeq \Z/{y^a \Z}$ and $
(\T/{I_M}) \otimes_{\Z} \Z_{y} \simeq (\T/J) \otimes_{\Z} \Z_{y} \simeq \Z/{y^a \Z}
$. 

Since the cuspidal divisor $C_{M, N}$ has order $m$ up to powers of 2 by Proposition \ref{cusp}, $\langle C_{M, N} \rangle$ has an order $y^c$ subgroup, say $D$. Because ${I_M}$ annihilates $C_{M, N}$ by Proposition \ref{cusp}, it also kills $D$. Thus, there is a natural surjection 
$$
\T/{I_M} \simeq \Z/{n\Z} \surj \End(D) \simeq \Z/{y^c \Z}.
$$
Therefore $y^c \mid n = y^a \times b$, so $c \leq a$. 

If $a=0$, then there is nothing to prove. Assume that $a>0$. 
We follow the argument in \cite[chap II, \textsection 5]{M77}.
Let $F:= \sum\limits_{n \geq 1} (T_n \pmod {J}) x^n$. It is the $x$-expansion of a cusp form $f$ over the ring $\Z/{y^a \Z}$. Note that $24f$ is a cusp form over $\Z/{24y^a\Z}$. 
Let $\nu=1$ if $y=3$, and $\nu=0$ otherwise.
Let $E_{M, N}$ be an Eisenstein series defined in \textsection \ref{sec:Eisensteinseries}. We divide into two cases.

\begin{enumerate}
\item Case 1 : 
Assume that $M=N$ and let $N=pA$.
Since $24E_{N, N}$ has an integral Fourier expansion at $i\infty$, $24E_{N,N} \pmod {y^{a+\nu}}$ is a modular form over $\Z/{y^{a+\nu}\Z}$. Thus, $24(f-E_{N, N}) \pmod {y^a}$ is a modular form over $\Z/{y^a\Z}$. 
But by Proposition \ref{Eisensteinseries} this is $\varphi(N) \pmod {y^a}$ (up to sign). 
Since $N$ is invertible in $\Z/{y^{a+\nu}\Z}$, by Ohta \cite[Lemma (2.1.1)]{Oh14}, there is a modular form $g$ of level $p$ over $\Z/{y^{a+\nu} \Z}$ such that
$g(Az)=24(f-E_{N, N}) \pmod {y^{a+\nu}}$. Note that the $x$-expansion of $g$ is a constant $\varphi(N)$ (up to sign). By the same argument in \cite[chap II, Proposition 5.12]{M77}, we get $g(z)=0$. 
Therefore $y^{a+\nu} \mid \varphi(N)$, which implies $a\leq c$. Thus, $a=c$ and we have
$$
(\T/{I_M}) \otimes_{\Z} \Z_{y} \simeq (\Z/{m\Z}) \otimes_{\Z} \Z_{y}. 
$$

\item Case 2 : 
Assume that $M\neq N$.
Since $g=(f-E_{M, N}) \pmod {y^{a+\nu}}$, which is a modular form over $\Z/{y^{a+\nu}\Z}$,
has a Fourier expansion at $i\infty$ equal to $0$, it is $0$ (on the irreducible component of $X_0(N)$ that contains $i\infty$) by the $q$-expansion principle \cite[\textsection 1.6]{Ka73}. 
Since $N$ is invertible in $\Z/{y^{a+\nu} \Z}$, 
the cusp $0$ belongs to the same component of $X_0(N)_{/\F_y}$ as $i \infty$, so the constant term of $g$ at the cusp $0$ is $0$.
Since $f$ is a cusp form, its constant term at the cusp $0$ is $0$ and hence the constant term of $E_{M, N}$ at the cusp $0$ is also $0$ modulo $y^{a+\nu}$. It is 
$-\frac{\varphi(N)\psi(N/M)}{24N(N/M)}$ by Proposition \ref{Eisensteinseries}.
Thus, $y^{a+\nu} \mid m = 24 y^c \times d$ and hence $a\leq c$. So, $a=c$ and we have
$$
(\T/{I_M}) \otimes_{\Z} \Z_{y} \simeq (\Z/{m\Z}) \otimes_{\Z} \Z_{y}. 
$$

\end{enumerate}
\end{proof}

\subsection{New Eisenstein ideals}
Let $N := pq$. (Since we assume that $N$ is square-free, $p\neq q$.)
On the new quotient of $J_0(N)$, $U_p$ acts by an involution, which is equal to $-w_p$. Thus, the possible eigenvalues of $U_p$ are either 1 or $-1$. Let
$$
I : = (U_p-1, ~U_q+1, ~T_r-r-1 ~:~ \text{for all primes}~ r \nmid N )
$$
be a ``new" Eisenstein ideal of level $N$.

In this case, we can compute the index of $I$ up to powers of 2. More specifically, let $\T/I \simeq \Z/{n\Z}$ for some $n$. Let $m$ be the numerator of $\frac{q+1}{(p(p+1), 3)}$.

\begin{thm}\label{thm:neweisen}
For any odd prime $y$, we get
$$
(\T/I) \otimes_{\Z} \Z_{y} \simeq (\Z/{m\Z}) \otimes_{\Z} \Z_{y}.
$$ 
\end{thm}

\begin{proof}
Let $n=y^a \times b$ and $m=y^c \times d$, where $(y, bd)=1$.
Let $J = (y^a, I)$. Then $\T/J \simeq \Z/{y^a \Z}$. 
Note that the order of the cuspidal divisor $\frac{(p-1)(q-1)}{3^t}C_{p, pq} =\frac{(p-1)(q-1)}{3^t}(P_1-P_p)$ in $J_0(pq)$ is the numerator of $\frac{q+1}{3^{1-t}}$ up to powers of $2$, where $t=1$ if $y=3$ and $p\equiv 1 \pmod 3$, and $t=0$ otherwise. Hence 
$\langle C_{p, pq} \rangle$ contains a subgroup of order $y^c$ and it is annihilated by $J$. Thus, there is a natural surjection
$$
\T/J \simeq \Z/{y^a \Z} \surj \End(D) \simeq \Z/{y^c \Z}.
$$
Therefore we get $c \leq a$. 

If $a=0$ there is nothing to prove. Assume that $a>0$. 
Then, $\m:=(y, I)$ is a maximal ideal of $\T(pq)$. By \cite[Lemma 2.1]{Yoo3}, $U_q \pmod {\m}$ is either $1$ or $q$. Therefore $q \equiv -1 \pmod {y}$. 

Let $F = \sum\limits_{n \geq 1} (T_n \pmod {J}) x^n$. It is the $x$-expansion of a cusp form $f$ over the ring $\Z/{y^a \Z}$. Consider $g=24(f-E_{p,pq}) \pmod {24y^a} =24\sum\limits_{n\geq 1} a_n x^n$, where $E_{p,pq}$ is the Eisenstein series in \textsection \ref{sec:Eisensteinseries}. 

\begin{enumerate}
\item
Suppose that $y\nmid 6p$. Let $G:=g \pmod {y^a}$.
Note that $pq$ is invertible in $\Z/{y^a\Z}$ and $a_n=0$ for $(n, q)=1$. Thus, by Ohta \cite[Lemma (2.1.1)]{Oh14}, there is a modular form $H$ of level $p$ over the ring $\Z/{y^a \Z}$, such that $H(q\tau) = G(\tau)= 24\sum\limits_{n\geq 1} a_n x^n$. Computing its coefficients, we get 
$H(\tau)=-24(q+1)\sum c_n x^n$, where $c_r=r+1$ for primes $r \nmid N$, $c_1=1$, $c_p=1$, and $c_q = q-1$. 
This is also a cuspidal eigenform of weight 2 for the Hecke operators $T_r$, for all primes $r \neq q$. By Ribet \cite[p. 491]{Wi95}, we get $\T(p) \otimes_{\Z} \Z_y=\T(p)^{(q)}\otimes_{\Z} \Z_y$, where $\T(p)^{(q)}$ is the $\Z$-subalgebra of $\End(J_0(p))$ generated by all $T_n$ with $(n, q)=1$. Therefore $H$ is a cuspidal eigenform of weight 2 for all the Hecke operators.

Suppose that $a>c$, i.e., $24(q+1)$ is not divisible by $y^a$. 
Then, $h:=H(\tau)/{y^c} \pmod y$ is a non-zero mod $y$ cuspidal eigenform of weight 2 and level $p$. Let $\m$ be the maximal ideal associated to $h$ and let $\rho_{\m}$ be the two dimensional semisimple representation of $\GQ$ associated to $\m$. Then, by the Brauer-Nesbitt theorem and Chebotarev density theorem, $\rho_{\m} \simeq 1\oplus \chi_{y}$, where $\chi_y$ is the mod $y$ cyclotomic character. By Mazur, this implies that
$p\equiv 1 \pmod y$ and there is a cuspidal eigenform $\eta$ (over $\Qbar$) of weight $2$ and level $p$ such that $\eta \equiv E_{p, p} \pmod y$. 

By the above lemma of Ribet and the duality, if two cuspidal eigenforms of weight $2$ and level $p$ (in this case over $\F_y$) have the same Fourier coefficients for all $n$ with $(n, q)=1$, then they have the same Fourier expansions. 
Let $x:=-24(q+1)/{y^c}$, which is a unit in $\Z_{y}$ by our assumption. Then, $h$ and $x\eta$ satisfy the condition. 
Therefore the $q^{\text{th}}$ coefficients are equal, i.e., $x(q-1)\equiv x(q+1) \pmod y$, which is a contradiction because $y$ is a prime greater than 3. Thus, $a\leq c$ as desired.

\item
Suppose that $y=3$ and $y \neq p$. Let $G:= g \pmod {3^{a+1}}$.
Note that the level $pq$ is invertible in $\Z/{3^{a+1}\Z}$ and $a_n=0$ for $(n, q)=1$. By the same argument as above, the $x$-expansion of $G$ is $0$. By the $q$-expansion principle \cite[\textsection 1.6]{Ka73}, we have $G=0$. Hence, the constant term of the Fourier expansion of $G$ at $P_p$ is $0$ modulo $3^{a+1}$. 
By the same computation of Proposition \ref{FJ}, it is $-\frac{(p-1)(q^2-1)}{q^2}$. Therefore $3^{a+1}$ divides $(p-1)(q+1)$ and $24(q+1)$. If $p \equiv 1 \pmod 3$, then $m=q+1$; and otherwise $m=\frac{q+1}{3}$. In both cases, we get $a \leq c$.

\item
Suppose that $y=p\geq 3$. Then, $\m$ is neither $p$-old nor $q$-old by Mazur \cite{M77} because $(p-1)(q-1)\not\equiv 0 \pmod y$ . Hence it is new. Thus, we have $(\T/I)\otimes \Z_y \simeq  (\T^{\new}/I)\otimes \Z_y$, where $\T^{\new}$ is the new quotient of $\T$. 

Let $\mathcal{T}(pq)$ be the $\Z$-subalgebra of $\End(J_0(pq))$ generated by $w_p$, $w_q$ and $T_r$ for primes $r \nmid pq$. Since $J_0(pq)^{\new}$ is stable under $\mathcal{T}(pq)$, we can define $\mathcal{T}^{\new}$ to be the image of $\mathcal{T}(pq)$ in $\End(J_0(pq)^{\new})$.
Since $U_p+w_p=U_q+w_q=0$ in $\End(J_0(pq)^{\new})$ by Lemma \ref{lem:hecke atkin}, $\cT^{\new}$ and $\T^{\new}$ are isomorphic, and 
$$
\cT_1:=\cT^{\new}/(w_p+1, w_q-1) \simeq \T^{\new}/(U_p-1, U_q+1).
$$

Let $\mathcal{I}=(T_r-r-1 : \text{ for primes } r \nmid pq) \subseteq \mathcal{T}(pq)$. Let 
$\mathcal{T}_0 := \mathcal{T}(pq) / (w_p+1,~w_q-1)$ and let $\cI_0$ be the image of $\cI$ in $\mathcal{T}_0$. 
Then, there is a natural surjection $\phi$ from $\cT_0$ to $\cT_1$, and let $\cI_1$ be the image of $\cI_0$ by $\phi$. Then,
$$
\xymatrix{
\cT_0/{\cI_0}  \otimes \Z_y 
\ar@{->>}[r]^-{\phi} & \cT_1/{\cI_1} \otimes \Z_y 
\ar[r]^-{\simeq} & (\T^{\new}/I) \otimes \Z_y \ar[r]^-{\simeq} & (\T/I)\otimes \Z_y.
}
$$

Now, we recall the result by Ohta. By \cite[Theorem (3.1.3)]{Oh14},
$$\mathcal{T}_0/\mathcal{I}_0 \otimes \Z_y \simeq (\Z/B\Z) \otimes \Z_y,$$ 
where $B$ is the numerator of $\frac{(p-1)(q+1)}{3}$. 
Therefore $y^a$ divides $B$, which implies that $a \leq c$. 
\end{enumerate}

In conclusion, we get $a \leq c$ in all cases and hence $a=c$ as desired.
\end{proof}

By the same argument as above (in particular, the method used in the proof of Case 1), we can prove the following. 
\begin{thm}
Let $N=Mq$, $(M,q)=1$, and $y \nmid 6N$.
Let $
I=(U_p-1, ~U_q+1, ~T_r-r-1 ~:~ \text{for all primes}~ p \mid M~ \text{and}~ r \nmid N ).
$
Then, we have
$$
(\T/I) \otimes_{\Z} \Z_{y} \simeq (\Z/{(q+1)\Z}) \otimes_{\Z} \Z_{y}.
$$
\end{thm}

\section{Multiplicity one}\label{sec:multi}
\subsection{Square-free level}
As before let $N=\prod\limits_{i=1}^{t} p_i$ and $M=\prod\limits_{j=1}^{s} p_j$ for some $1 \leq s \leq t$.
Let $\T:=\T(N)$ and $J:=J_0(N)$. For an ideal $I \subseteq \T$, we define the kernel of $I$ for $J$ as follows:
$$
J[I] := \{ x \in J_0(N)(\overline{\Q}) : Tx = 0 ~~\text{for all}~~ T\in I \}.
$$
Since $(\Ta_{\m} J) \otimes \Q$ is free of rank $2$ over $\T_{\m} \otimes \Q$ (cf. \cite[chap II, Lemma 7.7]{M77}) where $\Ta_{\m} J:= \lim\limits_{\leftarrow n} J[\m^n]$, $J[\m] \neq 0$ for any maximal ideal $\m$. Moreover, we get $\dim J[\m] \geq 2$ by Nakayama's lemma. 

As we explained in the introduction, for a non-Eisenstein maximal ideal $\m$, there is a notion of multiplicity. On the other hand, there is no natural one for an Eisenstein ideal $\m$. Instead, we define it as follows.
\begin{defn}
A multiplicity one theorem holds for $\m$ if $\dim_{\T/{\m}} J[\m]=2$. 
\end{defn}
In contrast to the non-Eisenstein case, the multiplicity one question for an Eisenstein ideal $\m$ has not been discussed much before. Mazur \cite{M77} proved that when $N$ is prime, $J[\m] \simeq \zell \oplus \muell$ for an Eisenstein maximal ideal $\m$ of residue characteristic $\ell \geq 3$.

Assume that $\ell>3$ is a prime and $(\ell, N)=1$. Let $\m:=(\ell, I_M)$, where
$$
I_M=(U_{p_i}-1, ~U_{p_j}-p_j, ~T_r-r-1 ~:~ 1 \leq i \leq s, ~s< j \leq t,~\text{for all primes}~ r \nmid N ).
$$
By the result of Theorem \ref{index}, $\m$ is maximal if and only if $\ell \mid \varphi(N)\psi(N/M)$. Thus, we assume that $\ell \mid \varphi(N)\psi(N/M)$. If $p_j \equiv 1 \modl $ for some $s < j \leq t$, then 
$\m = (\ell, I_{M\times p_j})$. Thus, we further assume that $p_j \not\equiv 1 \modl$ for all $s< j \leq t$. 
So, we have $s(\m)=s$, $1 \leq s(\m) \leq t$, and $0\leq s_0(\m) \leq s(\m)$.
(For the definition of notation, see the introduction.)

Let $S_N$ be the set of prime divisors of $N$ and let $S_{\m}$ be the set of primes at which $J[\m]$ is ramified.
Then, $S_{\m} \subseteq S_N \cup \{\ell \}$ by Igusa \cite{Ig59} and $ \varpi_{\ell}(S_{\m})\leq  \varpi_{\ell}(S_N)$.

\begin{thm}[Multiplicity one]\label{multi}
Assume one of the following.
\begin{enumerate}
\item $\varpi_{\ell}(S_N)= 1$.
\item $t=s+1$ and $\ell \nmid \varphi(N)$.
\end{enumerate}
Then a multiplicity one theorem holds for $\m$, i.e., $J[\m]$ is of dimension 2 over $\T/{\m}$.
\end{thm}

\begin{proof}
It suffices to show that $\dim J[\m] \leq 2$ because $\dim J[\m] \geq 2$.
We follow Mazur's idea in his paper \cite{M77} to analyze $J[\m]$. 
\begin{enumerate}
\item 
Assume that $\varpi_{\ell}(S_N)=1$. Let $p$ be a prime divisor of $N$ such that $p \equiv \pm 1 \modl$. We divide into three cases.
\begin{enumerate}
\item Case 1 : 
Assume that $s_0(\m)=0$, i.e., that $p\equiv -1 \modl$. 
Note that all Jordan-H\"older factors of $J[\m]$ are either $\zell$ or $\muell$ (cf. \cite[chap II, Proposition 14.1]{M77}). 
Moreover, $J[\m]$ can have at most one $\zell$ as its Jordan-H\"older factor by the $q$-expansion principle.
(cf. \cite[Chap II, Corollary 14.8]{M77}, note that $T_{\ell}-1 \in \m$, hence, it is ordinary. See also \cite[Lemma 2.7]{CS08}.) 
Since $C_{M, N} \in J(\Q)$, $\langle C_{M, N} \rangle [\m] \simeq \zell$, so we have $\zell \subseteq J[\m]$. Thus 
$$
0 \rightarrow \zell \rightarrow J[\m] \rightarrow B \rightarrow 0,
$$
where $B$ is a multiplicative group such that all its Jordan-H\"older factors are $\muell$. Since $B$ is annihilated by $T_r-r-1$ for all but finitely many primes $r$, by the theorem of constancy \cite[chap I, Lemma 3.5]{M77}), $B \simeq \muell^{\oplus a}$ for some $a\geq 1$. Since the Shimura subgroup $\Sigma_N$ is a maximal multiplicative subgroup of $J$ by Vatsal \cite[Theorem 1.1]{Va05} and $\Sigma_N[\m]=0$ (because $U_p$ acts as $p$ on $\Sigma_N$, $U_p-1 \in \m$ and $p\not\equiv 1 \modl$), we get $\muell \nsubseteq J[\m]$. 
Let $S_0 = S_N \cup \{\ell\}$, and let
$E:=\Ext_{S_0}(\muell, \zell)$ be the group of extensions of $\muell$ by $\zell$ over $\Z_{S_0}$. 
By Brumer-Kramer \cite[Proposition 4.2.1]{BK14}, the dimension of $E$ over $\F_{\ell}$ is $\varpi_{\ell}(S_0)=\varpi_{\ell}(S_N)$, which is $1$ by assumption. (It is generated by a non-trivial extension only ramified at $\ell$ and $p$.) 
Assume that $\dim J[\m] \geq 3$, then it has a submodule $V$ of dimension 3 that is a nontrivial extension of $\muell \oplus \muell$ by $\zell$. 
Let $\alpha$ (resp. $\beta$) be a natural inclusion of $\muell$ into the first (resp. second) component of $\muell\oplus \muell$
$$
\xymatrix{
0 \ar[r] & \zell \ar[r] & V \ar[r] & \muell \oplus \muell \ar[r] & 0 \\
0 \ar[r] & \zell \ar[r]\ar@{=}[u] & W \ar[r]\ar@<.5ex>[u]^-{\alpha}\ar@<-.5ex>[u]_-{\beta} & \muell \ar[r] \ar@<.5ex>[u]^-{\alpha}\ar@<-.5ex>[u]_-{\beta} & 0.
}
$$
Then $\alpha^*V$ and $\beta^*V$ are two elements in $E$, which is of dimension $1$. Thus, there are $a, b \in \F_{\ell}$ such that
$a \alpha^*V + b \beta^*V = 0$. Let $\gamma = a\alpha + b\beta$, then $\muell \subseteq \gamma^*V \subseteq V$, which is a contradiction.
Therefore $\dim J[\m]=2$. Note that since $J[\m]$ is a non-trivial extension, it is ramified at $p$.

\item Case 2 : 
Assume that $s_0(\m)=1$ but $s(\m)=s>1$, i.e, $p=p_1\equiv 1\modl$ and $p_2 \not\equiv 1\modl$. Then, $\Sigma_N[\m]=0$ because $U_{p_2}$ acts by $p_2$ on $\Sigma_N$. Therefore the same argument as above holds. 

\item Case 3 : 
Assume that $s(\m)=s_0(\m)=1$. Let $p=p_1\equiv 1 \modl$. Note that $\Sigma_N[\m] \simeq \muell$,
hence, $\muell \subseteq J[\m]$ but $\muell \oplus \muell \nsubseteq J[\m]$ from the assumption. 
Let $J[\m]$ be an extension of $\muell^{\oplus a}$ by $\zell$ for some $a\geq 1$.
Let $I_p$ be an inertia subgroup of $\GQ$ at $p$.
By a well known theorem of Serre-Tate \cite{ST68}, the kernel of $\m$ in the mod $p$ reduction of $J$ may be identified with $J[\m]^{I_p}$, the group of $I_p$-invariants. 
Let $J^0$ and $\Phi_p$ be the identity component and the component group of $J_{/\F_p}$, respectively.
Since $\ell \nmid \varphi(N/p)$ and $\Phi_p$ is generated by the image of $C_{p, N} = P_1 - P_p$ up to 2-, 3-primary groups, $\langle C_{p, N} \rangle[\m] \simeq \zell \subseteq J[\m]^{I_p}$ maps isomorphically into $\Phi_p[\m]$. (See Remark \ref{cuspcomp}.) 

Let $\gamma \in I_p$. Since $J[\m]$ is an extension of $\muell^{\oplus a}$ by $\zell$, for any $x \in J[\m]$ we get $(\gamma-1)x \in \zell \inj \Phi_p$.
On the other hand, by Grothendieck \cite{Gro72} (cf. \cite[\textsection 5]{MR91}), the group of $\ell$-torsion points of $J$, $J[\ell]$, has the following filtration. 
$$
J[\ell]^t \subseteq J[\ell]^f \subseteq J[\ell],
$$
where $J[\ell]^f$ can be regarded as a lift of $J^0[\ell]$. Moreover, the quotient $J[\ell]/J[\ell]^f$ is unramified. Now, we can think of $J[\m]$ as a subgroup of $J[\ell]$. Then, 
for any $x \in J[\m]$ we get $(\gamma-1) x = 0 \in J[\ell]/J[\ell]^f$ and hence $(\gamma-1)x \in J[\ell]^f \simeq J^0[\ell]$. 
Therefore we get $(\gamma-1)x \in J^0 \cap \Phi_p = 0$, which implies that $J[\m]$ is unramified at $p$. Thus, we get $p \not\in S_{\m}$ and $\varpi_{\ell}(S_{\m})=0$. 

If $\dim J[\m] \geq 3$, then it contains a non-trivial extension of $\muell$ by $\zell$ over $\Z_{S_{\m}}$. However, the dimension of $\Ext_{S_{\m}}(\muell, \zell)$ is $\varpi_{\ell}(S_{\m})=0$, which is a contradiction. Hence $\dim J[\m]=2$ and $J[\m] \simeq \zell \oplus \muell$. 
\end{enumerate}

\item 
Let $q = p_t$ for simplifying notation. Since $\ell \mid \varphi(N)\psi(q)$, $q \equiv -1 \modl$.
As in \textsection 2.5, let $J^0$, $\Phi_q$, and $T$ denote the identity component,
the component group, and the torus of $J_{/{\F_q}}$, respectively. Then, by Proposition \ref{comp}, 
$\Phi_q[\m]=0$. Since on $T[\m]$, $\Frob_q$ acts by $qU_q \equiv 1 \modl$ and $q \equiv -1 \modl$,
$T[\m]$ cannot contain $\muell$. Thus, $\dim T[\m] \leq 1$. 
Since $\ell \nmid \varphi(N)$, the index of the ideal $I_{N/q}=(U_{p_i}-1, ~T_r-r-1 ~:~ 1 \leq i \leq s, ~\text{for all primes}~ r \nmid N/q )$ of $\T_{N/q}$ is prime to $\ell$. Therefore $J_0(N/q)^2[\m]=0$, which implies that
$J_{/{\F_q}}[\m] \simeq J[\m]^{I_q}$ is at most of dimension 1 over $\T/{\m} \simeq \F_{\ell}$.
Since $J[\m]$ is an extension of $\muell^{\oplus a}$ by $\zell$ for some $a \geq 1$, $J[\m]^{I_q}$ 
is at least of dimension $a$, i.e., $J[\m]$ is of dimension 2. 
\end{enumerate}
\end{proof}

\begin{rem}\label{rem}
Because we assume that $p_j \not \equiv 1 \modl$ for a prime $s<j \leq t$, unramifiedness of $J[\m]$ at $p$ in the third case of the proof of (1) follows from the assumption $s(\m)=s=1$ only.
\end{rem}

By using a similar argument as above we can prove a bound on the dimension of $J[\m]$.
\begin{thm} \label{failmulti}
We have
$$
\max\{1+\varpi_0(\m), 2\} \leq \dim J[\m] \leq 1+\varpi_0(\m)+\varpi_{\ell}(S_{\m}).
$$

\end{thm}

\begin{proof}
If $s_0(\m)< s(\m)$, $\Sigma_N[\m]=0$ by Proposition \ref{Shimura}. 
Thus, $\muell \nsubseteq J[\m]$ but $\zell \subseteq J[\m]$. Let
$$
\xymatrix{
0 \ar[r] & \zell \ar[r] & J[\m] \ar[r] & \muell^{\oplus a} \ar[r] & 0 \\
0 \ar[r] & \zell \ar[r]\ar@{=}[u] & W_k \ar[r] \ar[u]_-{i_k} & \muell \ar[r] \ar[u]_-{i_k} & 0,
}
$$
where $W_k$ is the pullback of $J[\m]$ by the map $i_k : \muell \rightarrow \muell^{\oplus a}$,
which is an embedding into the $k^{\rm{th}}$ component for $1 \leq k \leq a$.
Then $W_k$ is an extension in $E=\Ext_{S_{\m}}(\muell, \zell)$. If $a > \varpi_{\ell}(S_{\m})=\dim E$, the extensions 
$W_k$ for all $1 \leq k \leq a$ are linearly dependent over $\F_{\ell}$. Thus, $J[\m]$ contains $\muell$, which is a contradiction. Therefore $\dim J[\m] = 1+ a \leq 1 + \varpi_{\ell}(S_{\m})$.

If $s(\m)=s_0(\m)=s$, then $\varpi_0(\m)=s$ and we have
$$
\Sigma_N[\m] \simeq \displaystyle\bigoplus_{i=1}^s \Sigma_{p_i}[\m] \simeq \muell^{\oplus s}
$$
by Proposition \ref{Shimura}. Thus, $\dim J[\m] \geq 1+\varpi_0(\m)$. Let $J[\m] = \muell^{\oplus s} \oplus V$ for some $V$. Then $\muell \nsubseteq V$, 
$\zell \subseteq V$, and $V$ is an extension of $\muell^{\oplus a}$ by $\zell$ (over $\Z_{S_{\m}}$). 
By the same argument as above, if $\dim V > 1+ \varpi_{\ell}(S_{\m})$ then $V$ contains
$\muell$, which is a contradiction. Therefore $\dim J[\m] = s + \dim V \leq s+1+ \varpi_{\ell}(S_{\m})$.
\end{proof}

\subsection{More on level $pq$}
Let $N=pq$, $p=p_1$ and $q=p_2$. (Hence $t=2$.) Then $s=1$ or $s=2$. 
Let $S_N:=\{ p, q \}$, $\T:=\T(pq)$ and $J:=J_0(pq)$.
\subsubsection{Case $s=1$}
Since $s=1$, assume that $q \not\equiv 1 \modl$ and $\ell \mid (p-1)(q^2-1)$. Let $\m:=(\ell, ~U_p-1, ~U_q-q, ~T_r-r-1 \mid ~\text{for all primes}~ r \nmid pq )$.
\begin{thm}\label{s1}
Then, 
\begin{enumerate}
\item In all cases below, $J[\m]$ is unramified at $p$.
\item If $p\not\equiv 1 \modl$, $\dim J[\m]=2$. 
\item If $p \equiv 1 \modl$ and $q \not\equiv -1 \modl$, $\dim J[\m]=2$.
\item Assume that $p \equiv 1 \modl$ and $q \equiv -1 \modl$. 
\begin{enumerate}
\item If $J[\m]$ is unramified at $q$, then $\dim J[\m]=2$.
\item If $J[\m]$ is ramified at $q$, then $\dim J[\m]=3$.
\end{enumerate}
\end{enumerate}
\end{thm}

\begin{proof}
\begin{enumerate}
\item 
This follows from Remark \ref{rem}. 

\item If $p\not\equiv 1 \modl$, since $\ell \mid (p-1)(q^2-1)$, $q \equiv -1 \modl$.
The claim holds by Theorem \ref{multi}(2) since $\ell \nmid \varphi(pq)$.

\item 
Assume that $p \equiv 1 \modl$ and $q \not\equiv -1 \modl$. 
The claim holds by Theorem \ref{multi}(1) since $\varpi_{\ell}(S_N) = 1$.

\item Since $p \equiv 1 \modl$, $\Sigma_{pq}[\m]=\Sigma_{p}[\m] \simeq \muell$. Thus, 
$J[\m]$ contains $\muell$ but $\muell \oplus \muell \nsubseteq J[\m]$.

\begin{enumerate}
\item Since $J[\m]$ is unramified at both $p$ and $q$, it is unramified everywhere (except $\ell$).
Therefore it is a direct sum of $\zell$ and $\muell^{\oplus a}$ (cf. \cite[chap I, Proposition 2.1]{M77}). Hence $\dim J[\m]=2$. 

\item 
In this case, $s(\m)=1=s_0(\m)=\varpi_0(\m)$ and $\varpi_{\ell}(S_{\m}) \leq 1$ since $J[\m]$ is unramified at $p$.
By Theorem \ref{failmulti}, $\dim J[\m] \leq 3$.
We know that $J[\m]$ contains $\zell$ and $\muell$ from the cuspidal group and the Shimura subgroup, respectively. Hence, $\dim J[\m]^{I_q} \geq 2$.
Since $J[\m]$ is ramified at $q$, $\dim J[\m]= 3$. 
\end{enumerate}
\end{enumerate}
\end{proof}

\begin{rem}
Let $q\equiv -1 \modl$. 
Note that $J[\m]$ does not depend on $p$ if $p \not\equiv 1 \modl$. It is a (unique) non-trivial extension of $\muell$ by $\zell$ over $\Z[1/q]$.
\end{rem}

\begin{exam}
In the case (4), we can compute the dimension of $J[\m]$ using SAGE \cite{SAGE}. For primes $p, q <100$ and $\ell <10$, we get $\dim J[\m]=3$ only
when $(p, q, \ell)=(41, 19, 5)$, $(61, 79, 5)$, $(29, 97, 7)$, $(43, 13, 7)$ and $(43, 41, 7)$. Thus,
we know that $J[\m]$ is ramified at $q$ in each of those cases.
\end{exam}

\begin{rem}
The structure of $J_0(43\times 13)[\m]$ for an Eisenstein $\m$ of residue characteristic $7$ is studied by Calegari and Stein \cite{CS08}.
We proved their result about its ramification (at $13$) from the dimension computation.
By Theorem \ref{s1}(1), we know that it is unramified at $43$.
\end{rem}

\subsubsection{Case $s=2$}
Let $\m:=(\ell, ~U_p-1, ~U_q-1, ~T_r-r-1 ~:~ \text{ for all primes}~ r \nmid pq )$. Since $\m$ is maximal if and only if $\ell \mid (p-1)(q-1)$, assume that $p \equiv 1 \modl$. 
\begin{thm}\label{s2}
Then, 
\begin{enumerate}
\item If $q \not\equiv \pm 1 \modl$, $\dim J[\m]=2$ and $J[\m]$ is ramified at $p$. 
\item Assume that $q \equiv -1 \modl$. Then $J[\m]$ is ramified at $p$. 
\begin{enumerate}
\item If $J[\m]$ is unramified at $q$, then $\dim J[\m]=2$.
\item If $J[\m]$ is ramified at $q$, then $\dim J[\m]=3$. 
\end{enumerate} 
\item If $q \equiv 1 \modl$, then $\dim J[\m]$ is either $4$ or $5$.
\end{enumerate}
\end{thm}

\begin{proof}
Since $p\equiv 1 \modl$, there is the Eisenstein maximal ideal $\m_p$ of level $p$ of 
residue characteristic $\ell$ and $J_0(p)[\m_p] \simeq \zell \oplus \muell$ by Mazur
\cite[chap II, Corollary 16.3]{M77}.

\begin{enumerate}
\item 
If $q \not \equiv \pm 1\modl$, by Theorem \ref{multi}(1), $\dim J[\m]=2$ and 
$J[\m]$ is ramified at $p$ because it is a unique non-trivial extension of $\Ext_{\{ p \}}(\muell, \zell)$.

\item Assume that $q \equiv -1 \modl$. Then
$\Sigma_{pq}[\m]=0$, in other words, $\muell \nsubseteq J[\m]$.

\begin{enumerate}
\item 
Assume that $J[\m]$ is unramified at $q$.
In this case, $s(\m)=1<s_0(\m)=2$ and $\varpi_{\ell}(S_{\m}) \leq 1$. Therefore $\dim J[\m]=2$ by Theorem \ref{failmulti} and $J[\m]$ is ramified at $p$ (and $\ell$).

\item 
Assume that $J[\m]$ is ramified at $q$. 
Let $T$, $J^0$ be the torus, the identity component of $J_{/{\F_q}}$, respectively.  
Then there is an exact sequence
$$
\xymatrix{
0 \ar[r] & T \ar[r] & J^0 \ar[r]_-\pi & J_0(p)\times J_0(p) \ar[r] & 0 \\
& & A:=J_0(p) \times J_0(p) \ar[u]_-{\alpha_{/{\F_q}}} \ar[ur]_-g
}
$$
and $g=\pi \circ \alpha_{/{\F_q}}$ is 
$\mat {1}{\Ver}{\Ver}{1}$, where $\Ver$ is the Verschiebung morphism in characteristic $q$ \cite[Lemma 1.1]{R90m}. Note that $\alpha_{/\F_q}$ is equal to the map (in characteristic $q$) induced by $\gamma_q^*$ in \textsection \ref{old and new}.
By taking the kernel of $\m$, we get 
$$
\xymatrix{
0 \ar[r] & T[\m] \ar[r] & J^0[\m] \ar[r]_-\pi & J_0(p)\times J_0(p)[\m].
}
$$
Since $J_0(p)[\m_p]\simeq \zell \oplus \muell$, $A[\m]= \{ (x, -x) ~:~ x\in J_0(p)[\m_p] \}$ and $\Ver$ acts as $q$ on $\zell$, $\{(x, -x) ~:~ x \in \zell \}$ maps injectively by $g$ to $J_0(p)^2[\m]$. 
Thus, the image of $\pi$ is at least 1-dimensional. 

Since $J[\m]$ is ramified at $q$, $\m$ is $q$-new.
Note that $X=\Hom(T, \Gm)$ is of rank 1 over $\T^{q\hyp\new}$ in the sense of Mazur (cf. \cite[\textsection 2]{R99}). Thus, we get $T[\m] \neq 0$ and 
therefore $J^0[\m]$ is at least 2-dimensional. 
Since $J_{/{\F_q}}[\m] \simeq J[\m]^{I_q}$ by Serre-Tate and 
$J[\m]$ is ramified at $q$, 
$$
\dim J[\m] \geq \dim J[\m]^{I_q}+1 \geq \dim J^0[\m]+1 \geq 3.
$$

On the other hand, the dimension of $J[\m]$
is at most 3 by Theorem \ref{failmulti}, and hence it is $3$. 

Moreover if it is unramified at $p$, then $\varpi_{\ell}(S_{\m}) \leq 1$ and $\varpi_0(\m)=0$. So, by Theorem \ref{failmulti}, $\dim J[\m]=2$, which is a contradiction. Therefore $J[\m]$ is also ramified at $p$.
\end{enumerate}

\item Assume that $q \equiv 1 \modl$. Then $\Sigma_{pq}[\m] \simeq \Sigma_p[\m]\oplus \Sigma_q[\m]\simeq \muell\oplus \muell \subseteq J[\m]$. 
By the result of Ribet \cite{Yoo14a}, $\m$ is new. Thus, $J_{\new}[\m]$ is non-trivial,
where $J_{\new}$ is the new subvariety of $J$. By the same argument as in 
\cite[chap II, \textsection 7]{M77}, $\Ta_{\m} J_{\new} \otimes \Q$ is free of rank 2 over $\T^{\new} \otimes \Q$. 
Therefore, $J_{\new}[\m]$ is at least of dimension 2 by Nakayama's lemma. Since $U_p+w_p$ acts by 2 on $\Sigma_{q}[\m]$, it is an isomorphism because $\ell$ is odd. Moreover, 
since $U_p+w_p$ and $U_q+w_q$ annihilate $J_{\new}$, $\Sigma_{q}[\m] \cap J_{\new} =0$. Similarly, we have $\Sigma_{p}[\m] \cap J_{\new} =0$. Thus, $\dim J[\m] \geq 2+2=4$. By Theorem \ref{failmulti}, the result follows.

\end{enumerate}
\end{proof}

\begin{cor}
If $q \equiv -1 \modl$ but $q$ is not an $\ell^{\rm{th}}$ power modulo $p$, then
$\dim J[\m]=2$. 
\end{cor}
\begin{proof}
By Ribet \cite{Yoo14a}, if $q$ is not an $\ell^{\rm{th}}$ power modulo $p$, then $\m$ is not new.
Since $q \equiv -1 \modl$, there is no Eisenstein maximal ideal of level $q$, hence $\m$ is $q$-old. Therefore $J[\m]$ is unramified at $q$. Thus, the claim follows from the case (2)(a) of Theorem \ref{s2}.
\end{proof}

\begin{rem}
 In the case (2)(resp. (3)) of Theorem \ref{s2}, the computation with SAGE \cite{SAGE} suggests that $\dim J[\m]=2$ (resp. $\dim J[\m]=5$). Now this observation is proved by Ribet and the author \cite{RY14}. Namely, 
$\dim J[\m]=2$ if $q \not \equiv 1 \modl$ and $\dim J[\m]=5$ if $q \equiv 1 \modl$.
\end{rem}

\bibliographystyle{annotation}

\end{document}